\documentclass[11pt,reqno]{amsart}

\usepackage{amscd,amssymb,amsmath,amsthm}
\usepackage{graphicx}
\usepackage{color}
\usepackage{comment}

\newtheorem{thm}{Theorem}
\newtheorem{defn}{Definition}
\newtheorem{lemma}{Lemma}
\newtheorem{pro}{Proposition}
\newtheorem{rk}{Remark}
\newtheorem{cor}{Corollary}

\newtheorem{ex}{Example}

\numberwithin{equation}{section} 

\begin{document}

\title[Gibbs measure theory]{Kolmogorov extension theorem for non-probability measures on Cayley trees}

\author{F.~H.~Haydarov}

\address{New Uzbekistan University,
54, Mustaqillik Ave., Tashkent, 100007, Uzbekistan,}

\address{AKFA University, 264, Milliy Bog street,  Yangiobod QFY, Barkamol MFY,
      Kibray district, 111221, Tashkent region, Uzbekistan,}

\address{V.I.Romanovskiy Institute of Mathematics,
9, University str. Tashkent, 100174, Uzbekistan.}

\address{National University of Uzbekistan,
 University str 4 Olmazor district, Tashkent, 100174, Uzbekistan.}

\email {haydarov\_imc@mail.ru}

\begin{abstract}
In this paper, we shall discuss the extendability of probability and non-probability measures on Cayley trees to a $\sigma$-additive measure on Borel fields which has a fundamental role in the theory of Gibbs measures.
\end{abstract}
\maketitle



\section{Introduction}

In classical physics, the Kolmogorov extension theorem lays the foundation
for the theory of stochastic processes. It has been known for a long time
that, in its original form, this theorem does not hold in quantum mechanics (e.g., \cite{4}).

Foundations of measure theory provide little support for compositional reasoning. Standard formalizations of iterative processes prefer to construct a single monolithic sample space from which all random choices are made at once. The central result in this regard is the Kolmogorov's extension theorem, which identifies conditions under which a family of measures on finite subproducts of an infinite product space extend to a measure on the whole space. This theorem is typically used to construct a large sample space for an infinite iterative process when the behavior of each individual step of the process is known (see \cite{5, 6}).

If the real line $\mathbb{R}$ and the $\sigma$-field of all Borel sets, and each family of probability measures $\mu_n, n\in\mathbb{N}$ which satisfies consistency condition, Kolmogorov proved such an extension is uniquely possible. Later Bochner generalized the result from the infinite product of measurable spaces to the projective limit of measurable spaces \cite{7}. Bochner's formulation is applicable to the infinite product of measurable spaces as a special case. Namely, if domain of probability measures are a locally compact, $\sigma$-compact metric space with the $\sigma$-field of all Borel sets, Bochner proved the extendability and its uniqueness for a consistent family of probability measures. (Of course, this result can be applied for a special case of infinite product measurable space).

Instead of locally compactness and $\sigma$-compactness, if we assume that each $\Omega_n$ is a complete metric space, the extendability and its uniqueness is proved for a family of probability measures (e.g. \cite{8}). Above results are still valid even if we replace a sequence of measures by a family of measures indexed by elements of an ordered set.

Generally speaking, there are some versions of Kolmogorov's extension theorem for probability measures. But some problems reduced to Kolmogorov's theorem for non-probability measures and it is impossible to apply the theorem for any infinite measures.  In this paper, we give analogue of Kolmogorov's theorem for probability measures on Cayley trees and prove the theroem for a certain class of infinite measures.

\section{Cylindric sets on Cayley trees}

The Cayley tree $\Im^{k}=(V, L)$ of order $k \geq 1$ is an infinite tree, i.e. graph without cycles, each vertex of which has exactly $k+1$ edges. Here $V$ is the set of vertices of $\Im^{k}$ va $L$ is the set of its edges.


 Consider models where the spin takes values in the set $\Phi$ (finite or denumerable),
and is assigned to the vertices of the tree. For $A\subset V$ a
configuration $\sigma_A$ on $A$ is an arbitrary function $\sigma_A:A\to
\Phi$. Let $\Omega_A=\Phi^A$ be the set of all configurations
on $A$. A configuration $\sigma$ on $V$ is defined as a function
$x\in V\mapsto\sigma (x)\in \Phi$; the set of all configurations
is $\Omega:=\Phi^V$. We consider all elements of $V$ are numerated (in any order) by the numbers: $0,1,2,3,...$. Namely, we can write $V=\{x_0, x_1, x_2, ....\}$ (detail in \cite{11, 9, 10}).

 Let $\mathcal{X}_A$ be the indicator function. $\Omega$ can be considered as a metric space with respect to the metric $\rho: \Omega \times \Omega \rightarrow \mathbb{R}^{+}$ given by
$$
\rho\left(\left\{\sigma(x_n)\right\}_{x_n \in V},\left\{\sigma^{\prime}(x_n)\right\}_{x_n \in V}\right)=\sum_{n \geq 0} 2^{-n}\mathcal{X}_{\sigma(x_n)\neq\sigma^{\prime}(x_n)}
$$
(or any equivalent metric the reader might prefer, this metric taken from \cite{2}), and let $\mathcal{B}$ be the $\sigma$-field of Borel subsets of $\Omega$.

For each $m \geq 0$ let $\pi_m: \Omega \rightarrow\Phi^{m+1}$ be given by $\pi_m\left(\sigma_0, \sigma_1, \sigma_2, ...\right)=\left(\sigma_0, \ldots, \sigma_m\right)$ and let $\mathcal{C}_m=\pi_m^{-1}\left(\mathcal{P}\left(\Phi^{m+1}\right)\right)$, where $\sigma_i:=\sigma(x_i)$ and $\mathcal{P}\left(\Phi^{m+1}\right)$ is the family of all subsets of $\Phi^{m+1}$ (Cartesian product of $\Phi$). Then $\mathcal{C}_m$ is a field and each of the sets in $\mathcal{C}_m$ is open and closed set in the metric space $(\Omega, \rho)$; also $\mathcal{C}_m \subset \mathcal{C}_{m+1}$. Let $\mathcal{C}=\bigcup_{m \geq 0} \mathcal{C}_m$; then $\mathcal{C}$ is a field (the field of \textbf{cylinder sets}) and each of the sets in $\mathcal{C}$ is both open and closed.
Denote  $\mathcal{S}(\mathcal{C})$ - the smallest sigma field containing $\mathcal{C}$. Every element of $\mathcal{S}(\mathcal{C})$ is called ``\textbf{measurable cylinder}". Put
$$
\sigma^{(m)}(q)=\left\{\sigma\in\Omega :\;
\sigma\big|_{\{x_m\}}=q\in \Phi\right\}.
$$

\begin{defn}\label{coun} A measurable space $(X, \mathcal{E})$ is said to be countably generated if $\mathcal{E}=\sigma(\mathcal{I})$ for some countable subset $\mathcal{I}$ of $\mathcal{E}$.
\end{defn}

\begin{pro}\label{L 16.1} $\mathcal{B}=\mathcal{S}(\mathcal{C})=\mathcal{S}\left(\left\{\sigma^{(m)}(q): m \geq 0,\ q\in \Phi\right\}\right)$ and in particular if $|\Phi|<\infty$ then $(\Omega, \mathcal{B})$ is countably generated.
\end{pro}
\begin{proof} Let $\mathcal{O}$ be the set of all open subsets of $\Omega$. Then $\mathcal{C}$ (if $|\Phi|<\infty$ then $\mathcal{C}$ is a countable set) a base for the topology on the metric space $(\Omega, \rho)$. Also, since $\mathcal{C}$ is the field, each $O \in \mathcal{O}$ can be written as a union of elements from $\mathcal{C}$. Hence $\mathcal{O} \subset \mathcal{S}(\mathcal{C})$ and thus $\mathcal{B}=\mathcal{S}(\mathcal{O}) \subset \mathcal{S}(\mathcal{S}(\mathcal{C}))=\mathcal{S}(\mathcal{C})$, i.e., $\mathcal{B}=\mathcal{S}(\mathcal{C})$. Moreover, each element of $\mathcal{C}$ can be written as a finite intersection of elements from the set $\left\{\sigma^{(m)}(q): m \geq 0,\ q\in \Phi\right\}$ and it therefore follows that $$\mathcal{C} \subset \mathcal{S}\left(\left\{\sigma^{(m)}(q): m \geq 0,\ q\in \Phi\right\}\right).$$ This implies that $\mathcal{B}=\mathcal{S}\left(\left\{\sigma^{(m)}(q): m \geq 0,\ q\in \Phi\right\}\right)$.
\end{proof}

For a fixed $x^{0} \in V$ we put
$$
W_{n}=\left\{x \in V \mid d\left(x, x^{0}\right)=n\right\}, \quad V_{n}=\bigcup_{m=0} ^{n} W_{m}, $$
%
where $d(x, y)$ is the distance between the vertices $x$ and $y$ on the Cayley tree, i.e. the number of edges of the shortest walk (i.e., path) connecting vertices $x$ and $y$.



For any fixed configuration $\sigma_{A}\in\Omega_{A}, \ A\subset V$ we denote:
$$\bar{\sigma}_{A}:=\left\{\sigma\in\Omega :\;
\sigma\big|_{A}=\sigma_A\right\}.$$

\begin{cor}\label{muhim}
$\mathcal{B}=\mathcal{S}\left(\left\{\bar{\sigma}_{V_n}: n\in \mathbb{N}_{0}\right\}\right)$.
\end{cor}
\begin{proof} By Proposition \ref{L 16.1} we have
$$\bar{\sigma}_{V_n}=\bigcap_{s_i\in V_n}\sigma^{(s_i)}(\bar{\sigma}_{V_n}(s_i))\in \mathcal{S}(\{\sigma^{(m)}(q): m\geq 0, \ q\in\Phi\})=\mathcal{B}.$$
Then for all $n\in\mathbb{N}$ we obtain that $\bar{\sigma}_{V_n}\in\mathcal{B},$ i.e. $\mathcal{S}\left(\left\{\bar{\sigma}_{V_n}: n\in \mathbb{N}\right\}\right)\subseteq \mathcal{B}.$

 On the other hand, we show that  $\mathcal{B}\subseteq\mathcal{S}\left(\left\{\bar{\sigma}_{V_n}: n\in \mathbb{N}\right\}\right).$ Let $m_0\geq 0$ and we can find $n_0\in\mathbb{N}$ such that $x_{m_0}\in V_{n_0}$. If bases of cylinder sets $\bar{\omega}_{V_{n_0}}, \bar{\nu}_{V_{n_{0}}}$ coincide with each other only at $\{x_{m_0}\}$ and its value be $q_0\in \Phi$ then we obtain that  $$\sigma^{(m_0)}(q_0)=\bar{\omega}_{V_{n_0}}\cap\ \bar{\nu}_{V_{n_0}}\in\mathcal{S}\left(\left\{\bar{\sigma}_{V_n}: n\in \mathbb{N}\right\}\right).$$
 From $m_0$ and $q_0$ are arbitrary numbers and Proposition \ref{L 16.1} we can conclude that  $\mathcal{B}\subseteq\mathcal{S}\left(\left\{\bar{\sigma}_{V_n}: n\in \mathbb{N}\right\}\right).$
\end{proof}

Note that Corrollary \ref{muhim} is very important in the theory of Gibbs measures (see \cite{9, 10}) and a family of sets $\{V_n\}_{n=1}^{\infty}$ is also \textbf{cofinal} sets \cite{g}.

\section{Analogue of Kolmogorov's Extension Theorem}
In this section we give an analogue of Kolmogorov's extension theorem for probability measures on Cayley trees. Actually, we prove the theorem for probability measures on $V_n, n\in\mathbb{N}$ and it gives us more short and understandable proof of the theorem.

\begin{lemma}\label{1.2.8} \cite{1} Let $\mu$ be a finitely additive set function on a field $\mathcal{F}$. Assume that $\mu$ is continuous from above at the empty set, that is, if $A_1, A_2, \ldots, \in \mathcal{F}$ and $A_n \downarrow \emptyset$, then $\mu\left(A_n\right) \rightarrow 0$. Then $\mu$ is countably additive on $\mathcal{F}$.
\end{lemma}
\begin{lemma}\label{1.4.11} If $\mu$ is a $\sigma$-finite measure on $\mathcal{B}$, then for each $B \in \mathcal{B}$, $$\mu(B)=\sup \{\mu(K): K \subset B, K \ \textrm{compact} \}.$$
\end{lemma}
\begin{proof} First assume that $\mu$ is finite. Let $\mathcal{I}$ be the class of subsets of $\Omega$ having the desired property; we show that $\mathcal{I}$ is a monotone class. Indeed, let $B_n \in \mathcal{I}, B_n \uparrow B$. Let $K_n$ be a compact subset of $B_n$ with $\mu\left(B_n\right)$ $\leq \mu\left(K_n\right)+\varepsilon, \varepsilon>0$ preassigned. By replacing $K_n$ by $\bigcup_{i=1}^n K_i$, we may assume the $K_n$ form an increasing sequence. Then $\mu(B)=\lim _{n \rightarrow \infty} \mu\left(B_n\right)$ $\leq \lim _{n \rightarrow \infty} \mu\left(K_n\right)+\varepsilon$, so that
$\mu(B)=\sup \{\mu(K): K$ a compact subset of $B\}$,
and $B \in \mathcal{I}$. If $B_n \in \mathcal{I}, B_n \downarrow B$, let $K_n$ be a compact subset of $B_n$ such that $\mu\left(B_n\right) \leq \mu\left(K_n\right)+\varepsilon 2^{-n}$, and set $K=\bigcap_{n=1}^{\infty} K_n$. Then $K\subset B$ and
$$
\mu(B)-\mu(K)=\mu(B\setminus K) \leq \mu\left(\bigcup_{n=1}^{\infty}\left(B_n\setminus K_n\right)\right) \leq \sum_{n=1}^{\infty} \mu\left(B_n\setminus K_n\right) \leq \varepsilon ;
$$
thus $B \in \mathcal{I}$. Also, $\mathcal{I}$ is a monotone class containing all cylinders with single point base (one point set is always compact in any topological space). Hence by Proposition \ref{L 16.1}, we obtain that $\mathcal{I}$ contains all Borel sets, i.e. $\mathcal{I}=\mathcal{B}$.

If $\mu$ is $\sigma$-finite, each $B \in \mathcal{B}$ is the limit of an increasing sequence of sets $B_i$ of finite measure. Each $B_i$ can be approximated from compact sets (apply the previous argument to the measure given by $\mu_i(A)=\mu\left(A \cap B_i\right), A \in \mathcal{B}$) and the preceding argument that $\mathcal{B}$ is closed under limits of increasing sequences shows that $B \in \mathcal{B}$.
\end{proof}

We now consider the problem of constructing $\sigma$-finite measures on $\mathcal{B}$. The approach will be as follows: Let $\Lambda=\{x_0, x_1, ..., x_n\}$ be a finite subset of $V$, where $x_0<x_1<\cdots<x_n$ (some fixed total ordering on $V$). Assume that for each such $\Lambda$ we are given a probability measure $\mu_{\Lambda}$ on $\mathcal{B}_\Lambda$, where  $\mathcal{B}_\Lambda$ is the minimal $\sigma$-field of Borel subsets of $\Omega^\Lambda$ (the set of all configurations on $\Lambda$). We denote by $\mathcal{N}$ the set of all finite subsets of $V$.

\begin{thm} \label{cara}\cite{1}\textbf{(Carath\'{e}odory Extension Theorem)} Let $\mu$ be a measure on the field $\mathcal{F}_0$ of subsets of $\Omega$, and assume that $\mu$ is $\sigma$-finite on $\mathcal{F}_0$, so that $\Omega$ can be decomposed as $\bigcup_{n=1}^{\infty} A_n$, where $A_n \in \mathcal{F}_0$ and $\mu\left(A_n\right)<\infty$ for all $n$. Then $\mu$ has a unique extension to a measure on the minimal $\sigma$-field $\mathcal{F}$ over $\mathcal{F}_0$.
\end{thm}

\begin{defn}\label{Kolm} For each $\Lambda \in \mathcal{N}$ let $\mu_{\Lambda}$ be a measure. The family of measures $\left\{\mu_{\Lambda}\right\}_{\Lambda \in \mathcal{N}}$ is said to be \textbf{consistent (compatible)} if $\mu_{\Lambda}(F_{\Lambda})=\mu_{\Delta}(F_{\Delta})$ for all $F_{\Lambda}=F_{\Delta} \in \mathcal{B}_{\Lambda}$ whenever $\Lambda \subset \Delta$.\end{defn}

Let $\Lambda\in \mathcal{N}$ and $\Delta\subset \Lambda$. If $\mu_{\Lambda}$ is a  measure on $\mathcal{B}_{\Lambda}$, the projection of $\mu_{\Lambda}$ on $\mathcal{B}_{\Delta}$ is measure $\pi_{\Delta}\left(\mu_{\Lambda}\right)$ on $\mathcal{B}_{\Delta}$ defined by
$$
\left[\pi_{\Delta}\left(\mu_{\Lambda}\right)\right](B)=\mu_{\Lambda}\left\{\sigma \in \Omega_{\Lambda}: \ \sigma |_{\Delta} \in B\right\}, \ B \in \mathcal{B}_{\Delta}.
$$
Similarly, if $\mu$ is a  measure on $\mathcal{B}$, the projection of $\mu$ on $\mathcal{B}_{\Lambda}$ is defined by
$$
\left[\pi_{\Lambda}(\mu)\right](B)=\mu\left\{\sigma \in \Omega: \sigma_{\Lambda} \in B\right\}=\mu(\bar{\sigma} |_{\Lambda}=\sigma_{\Lambda} : \sigma_{\Lambda}\in B), \quad B \in \mathcal{B}_{\Lambda}.$$
The following theorem is known:
\begin{thm}\label{Kol}\cite{1}\textbf{(Kolmogorov Extension Theorem)} For each $t$ in the arbitrary index set $T$, let $\Omega_t$ be a complete, separable metric space, with $\mathcal{F}_t$ the class of Borel sets (the $\sigma$-field generated by the open sets).

Assume that for each finite nonempty subset $v$ of $T$, we are given a probability measure $P_v$ on $\mathcal{F}_v$. Assume the $P_v$ are consistent, that is, $\pi_u\left(P_v\right)=P_u$ for each nonempty $u \subset v$.

Then there is a unique probability measure $P$ on $\mathcal{F}=\prod_{t \in T} \mathcal{F}_t$ such that $\pi_v(P)=P_v$ for all $v$.
\end{thm}

\begin{pro}\label{metric space} $(\Omega, \rho)$ is a separable and complete metric space. \end{pro}

\begin{proof}
At first we prove that $(\Omega, \rho)$ is the complete metric space. Let $\sigma_i\in \Omega, i\in \mathbb{N}_0$ be a Cauchy sequence of configurations in $\Omega$.
 For each real number $\varepsilon>0$, there exists a natural number $n_0\in\mathbb{N}_0$ such that, for every natural numbers $n, m \geq n_0$, we have $\rho(\sigma_n, \sigma_m)<\varepsilon$. For a fixed vertex $x_j\in V$ we consider $\sigma_i(x_j), i\in \mathbb{N}_0$. We show that there exists $i_0\in \mathbb{N}$ such that $\sigma_s(x_j)=\sigma_t(x_j)$ for all $s,t \geq i_0$. Indeed, suppose that there is a subsequence  $\sigma_{i_q}(x_j)$ with $\sigma_{i_q}(x_j)\neq \sigma_{i_{q+1}}(x_j)$. We can choose $q^{'}\in\mathbb{N}$ with $q^{'}>j$ and $\varepsilon <\frac{1}{2^{j+1}}$. Then by Cauchy sequence we can see easily
$\sigma_{i_q}(x_j)$ and $\sigma_{i_{q+1}}(x_j)$ belong to the same cylinder, i.e. $\sigma_{i_q}(x_j)=\sigma_{i_{q+1}}(x_j)$ which contradicts to the assumption. Put $\sigma_{i_0}(x_j)=\alpha_j^{\ast}$ and $\sigma^{\ast}(x_j)=\alpha_j^{\ast}$. By Cauchy sequence, we can conclude that from a certain term, all terms of the sequence $\sigma_i, i\in \mathbb{N}_0$ belong to the same cylinder. Namely, the sequence converge to $\sigma^{\ast}$. Hence, $(\Omega, \rho)$ is the complete metric space.

Now we show that $(\Omega, \rho)$ is the separable metric space.  Let $\Lambda\in \mathcal{N}$ and $\Omega_{\Lambda}$ is the set of all configuration on $\Lambda$, then $\Omega_{\Lambda}$ is denumerable set. Also $\Omega^{\ast}:=\bigcup_{\Lambda\in\mathcal{N}}\Omega_{\Lambda}$ is also denumerable set. For each $\sigma\in \Omega$ and $\varepsilon>0$ we consider $O_{\varepsilon}(\sigma):=\{\omega\in \Omega | \ \rho(\omega, \sigma)<\varepsilon\}$. Clearly, $O_{\varepsilon}(\sigma)$ contains a cylinder with finite base. By definition of $\Omega^{\ast}$, the base of the cylinder belongs to $\Omega^{\ast}$, i.e. $O_{\varepsilon}(\sigma)\cap \Omega^{\ast}\neq\emptyset$.
\end{proof}
 \begin{thm}\label{main1}\textbf{(Analogue of Kolmogorov Extension Theorem)} Let $\mu_n:=\mu_{V_n}$ on $\mathcal{B}_{V_n}$, $n\in\{0,1,2,...\}$ be a probability measure and $\mu_n$ are consistent, that is, $\pi_{V_i}\left(\mu_j\right):=\pi_i\left(\mu_j\right)=\mu_i$ for each $i<j$. Then there is a unique probability measure $\mu$ on $\mathcal{B}$ such that $\pi_{\Lambda}(\mu)=\mu_{\Lambda}$ for all $\Lambda\in \mathcal{N}$.
\end{thm}
\begin{proof} We define the hoped-for measure on cylinders with base $V_n$ by $\mu\left(\bar{\sigma}_{V_n}\right)=\mu_n\left(\sigma_{V_n}\right)$.
At first, we show that this definition makes sense since a given measurable cylinder can be represented in several ways.
It is sufficient to consider dual representation of the same measurable cylinder in the form $\bar{\sigma}_{V_i}=\bar{\sigma}_{V_j}$ where $i<j$. Then by the consistency hypothesis, we obtain that $\mu_i\left(\sigma_{V_i}\right)=\left[\pi_i\left(\mu_j\right)\right]\left(\sigma_{V_i}\right)$. Also, by definition of projection one gets
$$\left[\pi_i\left(\mu_j\right)\right]\left(\sigma_{V_i}\right)=\mu_j\left\{\sigma \in \Omega_{V_j}: \sigma |_{V_i} =\sigma_{V_i}\right\}.$$
But the assumption  $\bar{\sigma}_{V_i}=\bar{\sigma}_{V_j}$  implies that if $\sigma \in \bar{\sigma}_{V_j}$, then $\sigma \in \bar{\sigma}_{V_i}$ iff $\sigma |_{V_i}=\sigma_{V_i}$, hence $\mu_i\left(\sigma_{V_i}\right)=\mu_j\left(\sigma_{V_j}\right)$, as desired.

Thus, $\mu$ is well-defined on measurable cylinders with base $V_n$ and it's easy to check that the class of measurable cylinders with base $V_n$.

Let $\mathcal{F}$ be the set of finite union of all measurable cylinders with base $W_n, n\in\mathbb{N}_0$. Then $\mathcal{F}$ is a field. Now if $\overline{\sigma}^{(q)}_{A_1}, \overline{\sigma}^{(q)}_{A_2}, \ldots, \overline{\sigma}^{(q)}_{A_m}$ are disjoint sets in $\mathcal{F}$, where bases of the cylinders are subset of $V_q$.  Thus, by definition of $\mu$ and $\mu_q$ is a measure we obtain
$$\mu\left(\bigcup_{i=1}^m \overline{\sigma}^{(q)}_{A_i}\right)=\mu_q\left(\bigcup_{i=1}^m \sigma^{(q)}_{A_i}\right) =\sum_{i=1}^m\mu_q\left( \sigma^{(q)}_{A_i}\right)=\sum_{i=1}^m \mu\left(\overline{\sigma}^{(q)}_{A_i}\right).
$$
Therefore $\mu$ is finitely additive on $\mathcal{F}$.

Now we show that $\mu$ is countably additive on $\mathcal{F}$. Let $\overline{\sigma}_{A_k}, k=1,2, \ldots$ be a sequence of measurable cylinders decreasing to $\emptyset$. If $\mu\left(\overline{\sigma}_{A_k}\right)$ does not approach 0, we have, for some $\varepsilon>0, \mu\left(\overline{\sigma}_{A_k}\right) \geq \varepsilon>0$ for all $k$. Suppose $\overline{\sigma}_{A_k}=\overline{\sigma}_{A^{(q_k)}_k}$; by tacking on extra factors, we may assume that the numbers $q_k$ increase with $k$.

It follows from Lemma \ref{1.4.11} that we can find a compact set $\sigma_{C^{(q_k)}_k} \subset \sigma_{A^{(q_k)}_k}$ such that  $\mu_{q_k}\left(\sigma_{A^{(q_k)}_k}-\sigma_{C^{(q_k)}_k}\right)<\varepsilon / 2^{k+1}$.  Put
$$\overline{\omega}_k=\overline{\sigma}_{C_1} \cap \overline{\sigma}_{C_2} \ \cap \ \cdots \ \cap \ \overline{\sigma}_{C_k}\subset \overline{\sigma}_{A_1} \cap \overline{\sigma}_{A_2} \ \cap \ \cdots \ \cap \ \overline{\sigma}_{A_k}=\overline{\sigma}_{A_k}.$$
Then
$$\mu\left(\overline{\sigma}_{A_k}-\overline{\omega}_k\right)=
\mu\left(\overline{\sigma}_{A_k} \cap \bigcup_{i=1}^k (\overline{\sigma}_{C_k})^{c}\right) \leq \sum_{i=1}^k \mu\left(\overline{\sigma}_{A_k} \cap (\overline{\sigma}_{C_i})^{c}\right) \leq $$ $$\leq\sum_{i=1}^k \mu\left(\overline{\sigma}_{A_i}-\overline{\sigma}_{C_i}\right)<\sum_{i=1}^k 2 \varepsilon^{i+1}<\frac{\varepsilon}{2}.$$
Since $\overline{\omega}_{k}\subset \overline{\sigma}_{C_k}\subset \overline{\sigma}_{A_k}$ we obtain $$\mu\left(\overline{\sigma}_{A_k}-\overline{\omega}_{k}\right)=\mu\left(\overline{\sigma}_{A_k}\right)-\mu\left(\overline{\omega}_{k}\right).$$
Consequently, $\mu\left(\overline{\omega}_{k}\right)>\mu\left(\overline{\sigma}_{A_k}\right)-\frac{\varepsilon}{2}$. In particular, $\overline{\omega}_{k}$ is not empty.

Now pick $\sigma^{(k)} \in \overline{\omega}_k, k=1,2, \ldots $. Let $\overline{\sigma}_{C^{q_1}}=\overline{\sigma}^{(q_1)}_{C^{q_1}}$ [note all $\overline{\omega}_k \subset \overline{\sigma}^{(q_1)}_{C^{q_1}}$]. Consider the sequence
$$
\left(\sigma_{x_1}^1, \ldots, \sigma_{x_{v_1}}^1\right), \quad\left(\sigma_{x_1}^2, \ldots, \sigma_{x_{v_1}}^2\right), \quad \left(\sigma_{x_1}^3, \ldots, \sigma_{x_{v_1}}^3\right), \ldots,
$$
that is, $\sigma^1(v_1), \sigma^2(v_1), \sigma^3(v_1), \ldots$, where $v_{i}=|V_{q_i}|$.
Since $\sigma^n(v_{q_1})$ belong to $C^{q_1}$, a compact subset of $\Omega_{V_{q_1}}$, we have a convergent subsequence $\sigma^{j}(v_{q_1})$ approaching some $\xi(v_{q_1}) \in C^{q_1}$. If $A_2^{\prime}=C^{q_2}$ (so $\overline{\omega}_k \subset A_2^{\prime}$ for $k \geq 2$), consider the sequence $\sigma^{r_{11}}(v_{q_2}), \sigma^{r_{12}}(v_{q_2}), \sigma^{r_{13}}(v_{q_2}), \ldots \in C^{q_2}$ (eventually), and extract a convergent subsequence $\sigma^{r_{21}}(v_{q_2}) \rightarrow \xi(v_{q_2}) \in C^{q_2}$.

Note that $\xi(v_{q_2})|_{V_{q_1}}=\xi(v_{q_1})$; as $n \rightarrow \infty$, the left side approaches $\xi(v_{q_2})|_{V_{q_1}}$, and as $\left\{r_{2 n}\right\}$ is a subsequence of $\left\{r_{1 n}\right\}$, the right side approaches $x_{v_1}$. Hence $\xi(v_{q_2})|_{V_{q_1}}=\xi(v_{q_1})$.
Continue in this fashion; at step $i$ we have a subsequence
$$
\sigma^{r_{i n}}(v_{q_i}) \rightarrow \xi(v_{q_i}) \in C^{q_i},  \text { and } \xi(v_{q_i})|_{V_{q_j}}=\xi(v_{q_j}) \ \text {for} \ j<i.
$$
Pick any $\nu \in \Omega$ such that $\nu |_{V_{q_j}}=\xi(v_{q_j})$ for all $j=1,2, \ldots$ (such a choice is possible since $\xi(v_{q_i})|_{V_{q_j}}=\xi(v_{q_j}), j<i)$. Then $\nu |_{V_{q_j}}\in C^{q_j}$, for each $j$; hence
$$
\nu \in \bigcap_{j=1}^{\infty} \overline{\sigma}_{C_j} \subset \bigcap_{j=1}^{\infty} \overline{\sigma}_{A_j}=\emptyset,
$$
a contradiction. Thus $\mu$ extends to a measure on $\mathcal{F}$, and by construction, $\pi_{V_n}(\mu)=\mu_n$ for all $n\in \mathbb{N}_0$.

Finally, if $\mu_1$ and $\mu_2$ are two probability measures on $\mathcal{F}$ such that $\pi_{V_n}(\mu_1)=\pi_{V_n}(\mu_2)$ for all $n\in \mathbb{N}_0$,
then for any $B^n \in \mathcal{F}_v$,
$$
\mu_1\left(\overline{\sigma}_{V_n}\right)=\left[\pi_{V_n}(\mu_1)\right]\left(\sigma_{V_n}\right)
=\left[\pi_{V_n}(\mu_2)\right]\left(\sigma_{V_n}\right)=\mu_2\left(\overline{\sigma}_{V_n}\right).
$$
By Carath\'{e}odory extension theorem and Corollary \ref{muhim} we can conclude that $\mu_1(A)=\mu_2(A)$ for all $A\in \mathcal{B}$. \end{proof}

\section{Kolmogorov's extension theorem for non-probability measures}

There are several versions of Kolmogorov's extension theorem for probability measures. But some problems reduced to Kolmogorov's theorem for non-probability measures (e.g., \cite{3}). Actually, we can not apply the theorem for any infinite measures. But in this section we give Kolmogorov's extension theorem for a certain class of infinite measures.

In this section, we use notations of the previous sections.
Note that $\Omega_n=\Omega_{V_n}$ and $\mathcal{B}_n$ is the $\sigma$-ring of all Borel sets of $\Omega_n$. Also, $\mu_n$ is a measure on $(\Omega_n, \mathcal{B}_n)$, $n\in \mathbb{N}$.

 \begin{pro}\label{3.1.} For a consistent family of finite measures $\left\{\mu_n\right\}_{n=1}^{\infty}$, the extension is uniquely possible.
\end{pro}
 \begin{proof}
 Put $c_n=\mu_n\left(\Omega_n\right)<\infty$. From the consistency condition, $c_n$ does not depend on $n$, so we put this common value as $c$. Then, the measures $\bar{\mu}_n=c^{-1}\mu_n$ form a self-consistent family of probability measures, then by Theorem \ref{main1}, $\left\{\bar{\mu}_n\right\}_{n=1}^{\infty}$ can be extended uniquely to a $\sigma$-additive measure $\bar{\mu}$. It is obvious that the $\sigma$-additive measure $c \bar{\mu}$ is a unique extension of measures $\left\{\mu_n\right\}_{n=1}^{\infty}$.
\end{proof}

Let $\left\{\mu_n\right\}_{n=1}^{\infty}$ be a consistent family of infinite measures. As defined in the previous sections, we put $\mathcal{C}=\bigcup_n \pi_n^{-1}\left(\mathcal{C}_n\right)$ and $\mathcal{B}$ is the $\sigma$-field generated by $\mathcal{C}$. Then, the family $\left\{\mu_n\right\}_{n=1}^{\infty}$ defines a finitely additive measure $\mu$ on $\mathcal{C}$.

Suppose that $A \in \mathcal{C}$ and $\mu(A)<\infty$, namely suppose that $A=\pi_n^{-1}\left(A_n\right)$, $A_n \in \mathcal{B}_n$ and $\mu_n\left(A_n\right)<\infty$, then the measures
\begin{equation}\label{e2.1}\mu_k^{(A)}\left(E_k\right)=\mu_m\left(\pi_{m}^{-1}\left(E_k\right) \cap \pi_{m}^{-1}\left(A_n\right)\right) \ \text{for} \ E_k \in \mathcal{B}_k, \ m \geq \max(n, k)\end{equation}
form a consistent family of finite measures, so that $\left\{\mu_k^{(A)}\right\}$ can be extended uniquely to a $\sigma$-additive measure $\mu^{(A)}$ on $\mathcal{B}$.

For $A^{\prime} \supset A$, it is easily seen that we have
$$ \mu^{(A)}(E)=\mu^{\left(A^{\prime}\right)}(E \cap A) \ \text {for} \ E \in \mathcal{B}.$$
Namely,
\begin{equation}\label{e2.2}\mu^{\left(A^{\prime}\right)}(E \cap A)=\mu(E \cap A) \ \text {for} \ E \in \mathcal{C}.\end{equation}
Put $$\mathcal{B}_0=\left\{B \in \mathcal{B}; \  \exists \ A_1, A_2, \ldots \in \mathcal{C}, \mu\left(A_n\right)<\infty, B \subset \bigcup_{n=1}^{\infty} A_n\right\}.$$
  In this definition, $\left\{A_n\right\}$ can be supposed to be mutually disjoint. We shall always impose this additional condition on $\left\{A_n\right\}$.

For $B \in \mathcal{B}_0$, we define a $\sigma$-additive measure $\mu^{(B)}$ as follows:
\begin{equation}\label{e2.3} \mu^{(B)}(E)=\sum_{n=1}^{\infty} \mu^{\left(A_n\right)}(E \cap B) \ \text { for } \ E \in \mathcal{B}.\end{equation}

This $\mu^{(B)}$ is a $\sigma$-additive measure because every term in the right side is so.

Now, we shall remark that the measure $\mu^{(B)}$ does not depend on the choice of $\left\{A_n\right\}$. Suppose that $B \subset \cup_n A_n$ and $B \subset \cup_n A_n^{\prime}$. Then, from the $\sigma$-additivity of the measure $\mu^{\left(A_n\right)}$, we have
$$
\mu^{\left(A_n\right)}(E \cap B)=\sum_{k=1}^{\infty} \mu^{\left(A_n\right)}\left(E \cap B \cap A_k^{\prime}\right)
$$
but in virtue of $(\ref{e2.2})$, the right side is equal to
$$
\sum_{k=1}^{\infty} \mu^{\left(A_n\right)}\left(E \cap B \cap A_k^{\prime} \cap A_n\right)=\sum_{k=1}^{\infty} \mu^{\left(A_k^{\prime} \cap A_n\right)}(E \cap B)
$$
Thus, the right side of $(\ref{e2.3})$ is equal to
$$
\sum_{n=1}^{\infty} \sum_{k=1}^{\infty} \mu^{\left(A_k^{\prime} \cap A_n\right)}(E \cap B).
$$
This assures the independence of $\mu^{(B)}$ from the choice of $\left\{A_n\right\}$.

If $B, B^{\prime} \in \mathcal{B}_0$ and $B \subset B^{\prime}$, we have
\begin{equation}\label{e2.4} \mu^{(B)}(E)=\mu^{\left(B^{\prime}\right)}(E \cap B) \ \textrm{for} \  E \in \mathcal{B}.\end{equation}
But even if $E \cap B \in \mathcal{C}$, there is a question whether the right side of (\ref{e2.4}) is equal to $\mu(E \cap B)$ or not. We shall discuss it later.

 If the measures $\left\{\mu_n\right\}_{n=1}^{\infty}$ can be extended to a $\sigma$-additive measure $\bar{\mu}$ on $\mathcal{B}$, we must have from $(\ref{e2.2})$
$$\bar{\mu}(E \cap A)=\mu^{(A)}(E) \ \text {for} \ E\in \mathcal{B}, \ A \in \mathcal{C}, \ \mu(A)<\infty.$$
So, also for $B \in \mathcal{B}_0$ we have
\begin{equation}\label{e3.5}
\bar{\mu}(E \cap B)=\sum_{n=1}^{\infty} \bar{\mu}\left(E \cap B \cap A_n\right)=\sum_{n=1}^{\infty} \mu^{\left(A_n\right)}(E \cap B)=\mu^{(B)}(E).
\end{equation}
From (\ref{e3.5}) we have
\begin{equation}\label{e3.6} \bar{\mu}(E)=\mu^{(\Omega)}(E)\ \textrm{for any} \ E \in \mathcal{B}, \end{equation}
thus we have $\bar{\mu}=\mu^{(\Omega)}$. Hence, we can conclude the following Proposition.
\begin{pro} \label{Prop 2.2.} If $\Omega \in \mathcal{B}_0$, the extension (if possible) is unique. \end{pro}

The possibility of extension depends on whether $\mu^{(\Omega)}$ is identical with $\mu$ or not on $\mathcal{C}$. Namely, the condition:
\begin{equation}\label{e2.7} \mu^{(\Omega)}(E)=\mu(E) \ \textrm{for} \ E \in\mathcal{C} \end{equation}
is necessary and sufficient for the unique extension of $\left\{\mu_n\right\}$ to a $\sigma$-additive measure. (\ref{e2.7}) means
\begin{equation}\label{e2.7.}  \mu(E)=\sum_{n=1}^{\infty} \mu\left(E \cap A_n\right) \ \textrm{for} \ E \in \mathcal{C},\end{equation}
where $\left\{A_n\right\}$ is such that $A_n \in \mathcal{C}, \mu\left(A_n\right)<\infty$ and $\Omega=\bigcup_{n=1}^{\infty} A_n$.

If all $A_n$ can be chosen in $\pi_m^{-1}\left(\mathcal{B}_m\right)$ for fixed $m$, the condition (\ref{e2.7.}) is satisfied in virtue of the $\sigma$-additivity of $\mu_m$. Therefore we have

\begin{thm}\label{Theorem 3.1.} If one of measures $\left\{\mu_n\right\}_{n=1}^{\infty}$, say $\mu_{n_0}$, is $\sigma$-finite, then $\left\{\mu_n\right\}_{n=1}^{\infty}$ can be extended uniquely to a $\sigma$-additive measure on $\mathcal{B}$.
\end{thm}

\section*{Acknowledgements}
The work supported by the fundamental project (number: F-FA-2021-425)  of The Ministry of Innovative Development of the Republic of Uzbekistan. I thank Professor U.A.Rozikov for useful discussions and suggestions
which have improved the paper.

\section*{Statements and Declarations}

{\bf	Conflict of interest statement:}
The author states that there is no conflict of interest.

\section*{Data availability statements}
The datasets generated during and/or analysed during the current study are available from the corresponding author on reasonable request.

\end{document}